\documentclass[10pt,letterpaper]{amsart}
\usepackage{fancyhdr}
\usepackage{enumitem}
\usepackage{amsmath,amsthm,amssymb,mathrsfs}
\usepackage{tikz-cd}
\usepackage{marginnote}
\newcommand{\filename}{Slope-K-stab-PPn-Calculations-11-Sept-2026.tex} 

\providecommand{\binom}[2]{{#1\choose#2}}

\renewcommand{\geq}{\geqslant}
\renewcommand{\leq}{\leqslant}
\newcommand{\Osh}{{\mathcal O}}                        

\newcommand{\Sym}{\operatorname{Sym}}

\newcommand{\K}{\mathrm{K}}      
                      
\newcommand{\kk}{\mathbf{k}}

\newcommand{\PP}{\mathbb{P}} 
\newcommand{\QQ}{\mathbb{Q}} 
\newcommand{\ZZ}{\mathbb{Z}} 

\newtheorem{theorem}{Theorem}[section]
\newtheorem{lemma}[theorem]{Lemma}
\newtheorem{corollary}[theorem]{Corollary}

\theoremstyle{definition}

\newtheorem{example}[theorem]{Example}

\numberwithin{equation}{section}

\title[Slope $\K$-semistability along a linear subspace]{Seshadri slope $\K$-semistability for the blow-up of projective space along a linear subspace}

\author{Nathan Grieve}

\address{
Department of Mathematics, National Taiwan University, Astronomy and Mathematics Building 5F, No. 1, Sec. 4, Roosevelt Rd., Taipei 10617, Taiwan (R.O.C.); School of Mathematics and Statistics, 4302 Herzberg Laboratories, Carleton University, 1125 Colonel By Drive, Ottawa, ON, K1S 5B6, Canada; 
D\'{e}partement de math\'{e}matiques, Universit\'{e} du Qu\'{e}bec \`a Montr\'{e}al, Local PK-5151, 201 Avenue du Pr\'{e}sident-Kennedy, Montr\'{e}al, QC, H2X 3Y7, Canada
}
\email{nathan.m.grieve@gmail.com}%

\begin{document}

\begin{abstract}
We study a particular instance of the frame work from \cite{Grieve:CM:Line:Slope:Stab}, building on earlier work of Arezzo-et-al \cite{Arezzo:DellaVedova:LaNave}, Ross and Thomas \cite{Ross:Thomas:2007}, \cite{Ross:Thomas:2006} and others, which applies the theory of the Chow-Mumford line bundle to determine the Donaldson-Futaki invariant of deformation to the normal cone test configurations with respect to big and nef line bundles.  In particular, here we consider the case of projective $n$-space blown-up along a linear subspace.  Our main result arises as an application of  \cite[Theorem 1.1]{Grieve:CM:Line:Slope:Stab}.  It establishes $\K$-semistability of the deformation to the normal cone test configuration for the blow-up $\PP(\Osh_{\PP^s}^{\oplus r} \oplus \Osh_{\PP^s}(1)) \simeq \operatorname{Bl}_{\PP^{r-1}}(\PP^s)$ along the exceptional divisor and with respect to the tautological line bundle $\Osh_{\PP(\Osh_{\PP^s}^{\oplus r} \oplus \Osh_{\PP^s}(1))}(1)$.

\medskip

\noindent\textsc{R\'esum\'e.}
Nous \'{e}tudions un cas particulier du cadre d\'efini dans \cite{Grieve:CM:Line:Slope:Stab}, en nous appuyant sur les travaux ant\'{e}rieurs d'Arezzo et al. \cite{Arezzo:DellaVedova:LaNave}, de Ross et Thomas \cite{Ross:Thomas:2007}, \cite{Ross:Thomas:2006}, et d'autres, qui applique la th\'eorie du fibr\'e en droites de Chow-Mumford pour d\'eterminer l'invariant de Donaldson-Futaki de la configuration test de d\'eformation vers le c\^one normal, relativement \`a des fibr\'es en droites big et nef. Nous consid\'erons ici, en particulier, le cas de l'espace projectif de dimension $n$ \'eclat\'e le long d'un sous-espace lin\'{e}aire. Notre r\'{e}sultat principal d\'{e}coule d'une application du \cite[Theorem 1.1]{Grieve:CM:Line:Slope:Stab}. Il \'etablit la $\K$-semi-stabilit\'e de la configuration test de d\'{e}formation vers le c\^one normal pour l'\'{e}clatement $\mathbb{P}(\mathcal{O}_{\mathbb{P}^s}^{\oplus r} \oplus \mathcal{O}_{\mathbb{P}^s}(1)) \simeq \text{Bl}_{\mathbb{P}^{r-1}}(\mathbb{P}^s)$ le long du diviseur exceptionnel, relativement au fibr\'{e} en droites tautologique $\mathcal{O}_{\mathbb{P}(\mathcal{O}_{\mathbb{P}^s}^{\oplus r} \oplus \mathcal{O}_{\mathbb{P}^s}(1))}(1)$.
\end{abstract}

\thanks{
The author thanks the Natural Sciences and Engineering Research Council of Canada for their support through his grants, DGECR-2021-00218 and RGPIN-2021-03821, and also the National Science and Technology Council (Taiwan) for their support through his grants 115-2115-M-002-003-MY3 and 115-2811-M-002-043.\\
ORCID: https://orcid.org/0000-0003-3166-0039
\\
Date: \today.  \\
File name: \filename}

\maketitle

\section{Introduction}

The recent article \cite{Grieve:CM:Line:Slope:Stab}, building on earlier work of Arezzo-et-al \cite{Arezzo:DellaVedova:LaNave}, Ross and Thomas \cite{Ross:Thomas:2007}, \cite{Ross:Thomas:2006} and others, applies the theory of the Chow-Mumford line bundle to determine the Donaldson-Futaki invariant of deformation to the normal cone test configurations with respect to big and nef line bundles.  Here, we apply that result, \cite[Theorem 1.1]{Grieve:CM:Line:Slope:Stab} and prove Theorem \ref{big:and:nef:blow:up:PP:n} below.  Here and elsewhere we refer to \cite{Grieve:CM:Line:Slope:Stab} for notation and definitions as they pertain to test configurations and deformation to the normal cone.  Throughout, we work over an algebraically closed characteristic zero base field $\kk$.

\begin{theorem}\label{big:and:nef:blow:up:PP:n}  
Let  $\pi \colon \PP(\mathcal{E}) \simeq \operatorname{Bl}_{\PP^{r-1}}(\PP^s) \rightarrow \PP^s$, for $\mathcal{E} := \Osh_{\PP^s}^{\oplus r} \oplus \Osh_{\PP^s}(1)$, be the blowing-up of $\PP^s$ along an $r-1$ plane.  Let $L_0 := \Osh_{\PP(\mathcal{E})}(1)$ and let $E$ be the exceptional divisor.  Let $(\mathcal{X},\mathcal{L}_c)=(\mathcal{X},\mathcal{L}_1)$ be the deformation to the normal cone test configuration for 
$$X := \PP(\mathcal{E}) := \operatorname{Proj} \left(\operatorname{Sym}^\bullet (\mathcal{E}) \right) \rightarrow \PP^s$$ 
with respect to $L_0$ and along $E$ with $c = \epsilon(L_0;E) = 1$.  Then $(\mathcal{X},\mathcal{L}_1)$ has Donaldson-Futaki invariant equal to zero.
\end{theorem}
\begin{proof}
Applying \cite[Theorem 1.1]{Grieve:CM:Line:Slope:Stab}, the conclusion is that the deformation to the normal cone test configuration has Donaldson-Futaki invariant equal to a positive multiple of the quantity \eqref{slope:stab:key:eqn}.  The conclusion of Theorem \ref{big:and:nef:blow:up:PP:n} thus follows as an application of Theorem \ref{instability:main:thm}.
\end{proof}

The conclusion of Theorem \ref{big:and:nef:blow:up:PP:n} can be expressed at the level of numerical classes.  This is achieved via Corollary \ref{big:and:nef:blow:up:PP:n:cor} below.

\begin{corollary}\label{big:and:nef:blow:up:PP:n:cor}
With the same hypothesis as Theorem \ref{big:and:nef:blow:up:PP:n}, let  $M$ be the numerical class of the tautological line bundle $\Osh_{\PP(\mathcal{E})}(1)$ and let $E$ be the numerical class of the exceptional divisor.  Then the big and nef class $M$ is Seshadri slope $\K$-semistable with respect to $E$.  It is not Seshadri slope $\K$-stable.
\end{corollary}

\begin{proof}
As is evident from their definitions in Section \ref{case:of:PPs:OOr:OOa}, the slope stability invariants that are used to determine the nature of the Donaldson-Futaki invariant that arises in Theorem \ref{big:and:nef:blow:up:PP:n} are well-defined at the level of numerical classes.  Thus the conclusion desired by Corollary \ref{big:and:nef:blow:up:PP:n:cor} follows as a direct consequence of Theorem \ref{big:and:nef:blow:up:PP:n}.
\end{proof}

\subsection*{Acknowledgements}
The author thanks colleagues for their interest, discussions and correspondence on related topics.  The author thanks the Natural Sciences and Engineering Research Council of Canada for their support through his grants, DGECR-2021-00218 and RGPIN-2021-03821, and also the National Science and Technology Council (Taiwan) for their support through his grants 115-2115-M-002-003-MY3 and 115-2811-M-002-043.  Some preliminary calculations which led to the discovery of Theorem \ref{big:and:nef:blow:up:PP:n} were conducted jointly with Ruiran Sun while he was a CRM postdoctoral fellow at McGill University and jointly supported by the author, Julien Keller and Steven Lu. 

\section{The case of $\PP(\Osh_{\PP^s}^{\oplus r} \oplus \Osh_{\PP^s}(a))$, for $a>0$}\label{case:of:PPs:OOr:OOa}

Fixing positive integers $r$ and $s$, together with a positive integer $a \in \ZZ_{>0}$, over projective space $\PP^s$, consider the vector bundle
$$
\mathcal{E} = \Osh_{\PP^s}^{\oplus r} \oplus \Osh_{\PP^s}(a) 
$$
together with its projectivization
$$
\pi \colon \PP(\mathcal{E}) := \operatorname{Proj}\left(\Sym^\bullet(\mathcal{E})\right) \rightarrow \PP^s \text{.}
$$
Then
$\dim \PP(\mathcal{E}) = n := s + r \text{.}$

Inside of the Chow ring
$A^\bullet(\PP(\mathcal{E})) = A^\bullet_{\QQ}(\PP(\mathcal{E})) \text{,}$
let $M$ denote the class of $\Osh_{\PP(\mathcal{E})}(1)$ and let $N$ denote the class of $\pi^*\Osh_{\PP^s}(1)$.  

The class of the canonical line bundle is then
$$
\K = -(r+1)M - (s+1-a)N \text{.}
$$
Put
$E := M - N \text{.}$
Fixing 
$\epsilon \in [0,1) \bigcap \QQ \text{,}$
set
$$L_{\epsilon} := M - \epsilon E = (1 - \epsilon)M + \epsilon N$$
and
$c := 1 - \epsilon \text{.}$
Then
$c = \epsilon(L_{\epsilon};E) \in \QQ \text{.}$

We want to describe in explicit terms the \emph{Seshadri slope $\K$-stability invariants}
$$
\mu_c(L_{\epsilon};E) := \frac{ \int_0^c \left(\alpha_1(t) + \frac{\alpha_0'(t)}{2} \right) \mathrm{d}t}{\int_0^c \alpha_0(t)\mathrm{d}t}
\text{ and }
\mu(L_{\epsilon}) := \frac{\alpha_1}{\alpha_0} \text{.}
$$

Here, for $0 \leq t \leq c$, we have defined
$$
\alpha_0(t) := \frac{1}{n!} \left( (L_{\epsilon}-tE)^n \right) \text{, }
\alpha_1(t) := - \frac{\left(\K \cdot (L_{\epsilon}-tE)^{n-1}\right)}{2(n-1)!}
$$
$$
\alpha_0 = \frac{(L_{\epsilon}^n)}{n!}
\text{ 
and
}
\alpha_1 = - \frac{(\K \cdot L_{\epsilon}^{n-1})}{2(n-1)!} \text{.}
$$

In calculating the slope stability invariants, working inside of the Chow ring, we make use of the fact that 
$$
A^\bullet(\PP(\mathcal{E})) \simeq A^\bullet(\PP^s)[M] / \langle M^{r+1} - a N M^r \rangle 
$$
where
$$A^\bullet(\PP^s) \simeq \QQ[N] / \langle N^{s+1} \rangle \text{.}$$
We refer to \cite[Section 3.3]{Fulton} or \cite[Section 9.3]{Harris:Eisenbud:3264} for more details.

Using these isomorphisms, it follows that
$$
\alpha_0(t) = \frac{1}{n!} \sum_{i=r}^n \binom{n}{i}(1-\epsilon - t)^i(\epsilon + t)^{n-i}a^{i-r}
$$

\begin{align*}
\alpha_0'(t) & = - \frac{1}{(n-1)!} \left( \sum_{i = r - 1}^{n-1} \binom{n-1}{i}(1-\epsilon - t)^i(\epsilon + t)^{n-1-i} a^{i-r+1}  \right. \\
&
\left. 
- \sum_{i=r}^{n-1} \binom{n-1}{i}(1 - \epsilon - t)^i (\epsilon + t)^{n-1-i} a^{i-r} \right)
\end{align*}
and

\begin{align*}
\alpha_1(t) & = 
- \frac{1}{2 (n-1)!} \left( - (r+1) \sum_{i = r - 1}^{n-1} \binom{n-1}{i}(1-\epsilon - t)^i(\epsilon + t)^{n-1-i}a^{i-r+1} \right. \\
&
\left. -(s+1-a) \sum_{i=r}^{n-1} \binom{n-1}{i}(1 - \epsilon - t)^i (\epsilon + t)^{n-1-i} a^{i-r} \right) \text{.}
\end{align*}

For later use, see Section \ref{Case:of:OOr:OO1}, we remark that in calculating the quantity
\begin{equation}\label{big:nef:slope}
\mu_1(L_0;E)  = \frac{ \int_0^1 \left( \alpha_0(L_0;E, t) + \frac{\alpha_0'(L_0;E,t)}{2} \right) \mathrm{d}t }{ \int_0^1 \alpha_0(L_0;E,t) \mathrm{d}t } 
\end{equation}
it is helpful to recall that for all nonnegative integers $k,\ell$
$$
\int_0^1(1-t)^kt^\ell \mathrm{d}t = \frac{k! \ell!}{(k+\ell+1)!} \text{.}
$$

For instance, it follows that 
$$
\int_0^1 \alpha_0(L_0;E,t) \mathrm{d}t = \frac{1}{(n+1)!} \sum_{i=r}^n a^{i-r}
$$

$$
\int_0^1 \alpha_1(L_0;E,t) \mathrm{d}t = \frac{1}{2n!} \left((r+1) \sum_{i = r-1}^{n-1} a^{i - r +1} + (n-r+1-a) \sum_{i=r}^{n-1} a^{i-r} \right)
$$
and
$$
\int_0^1 \frac{\alpha_0'(t)}{2}\mathrm{d}t = \frac{-1}{2n!} \left( \sum_{i=r-1}^{n-1} a^{i-r+1} - \sum_{i=r}^{n-1} a^{i-r} \right) \text{.}
$$

Let us now illustrate the above for some special values of $a$ and $r$.

\begin{example}\label{slope:stability:FFa:surfaces}
Here, we consider the special case that $r = 1$ and $a>0$.  In particular, we consider the case of the $\mathbb{F}_a$-surface 
$$
\pi \colon \mathbb{F}_a \simeq \PP(\mathcal{E}) := \operatorname{Proj} \left( \operatorname{Sym}^\bullet ( \mathcal{E} ) \right) \rightarrow \PP^1 \text{, for $a \in \ZZ_{>0}$.}
$$
Here $\mathcal{E} := \Osh_{\PP^1} \oplus \Osh_{\PP^1}(a) \text{.}$ Inside of the Chow ring $A^\bullet(\PP(\mathcal{E}))$, let $M$ denote the class of $\Osh_{\PP(\mathcal{E})}(1)$ and let $N$ denote the class of $\pi^* \Osh_{\PP^1}(1)$.  The canonical class is then
$$
\K = - 2 M + (a-2) N \text{.}
$$
Furthermore, the following relations hold true
$$
\text{$N^i = 0$, for $i > 1$, $M^2 = a M N$; and $M^i = 0$ for $i > 2$.}
$$

Put
$E := M - N$
and, for each 
$\epsilon \in [0,1) \bigcap \QQ \text{,}$
set
$$
L_\epsilon = M - \epsilon E 
 = (1-\epsilon) M + \epsilon N \text{;}
 $$
in particular
$M = L_ 0 $
and $c := \epsilon(L_\epsilon;E) = 1 - \epsilon \text{.}$ 

Let us now consider the question of \emph{Seshadri-slope $\K$-stability} of the big and nef class $L_\epsilon$ with respect to the exceptional class $E$.   (Note that $E$ is contracted under the map that is induced by the complete linear series $|\Osh_{\PP(\mathcal{E})}(1)|$.)

In particular, the \emph{numerical slope quantities} that of interest are respectively calculated to be
$$
\mu = \mu(L_{\epsilon}) 
= \frac{(1 - \epsilon)a + \epsilon - \frac{(a-2)(1-\epsilon)}{2} }{ \frac{(1-\epsilon)^2 a + 2 \epsilon (1 - \epsilon) }{2} } 
$$
which is the \emph{slope} of $X$ with respect to $L_{\epsilon}$ and
$$
\mu_{1 - \epsilon} = \mu_{1 - \epsilon}(L_{\epsilon};E) 
= \frac{\frac{\epsilon^2}{2} - \frac{3}{2}\epsilon + 1}{- \frac{1}{6}(a-2)\epsilon^3 + \frac{1}{2}(a-1)\epsilon^2 - \frac{1}{2}a\epsilon + \frac{1}{6}a+\frac{1}{6}} 
$$
which is the \emph{slope of $L_{\epsilon}$ along $E$}.  

To determine the extent to which $L_{\epsilon}$ is \emph{slope $\K$-semistable along $E$} and with respect to the Seshadri constant $\epsilon(L_\epsilon;E)$, it is necessary to study the extent to which the difference 
\begin{align*}
\mu_{1 -\epsilon} (L_{\epsilon};E)- \mu(L_{\epsilon}) & 
= \frac{\frac{\epsilon^2}{2} - \frac{3}{2}\epsilon + 1}{- \frac{1}{6}(a-2)\epsilon^3 + \frac{1}{2}(a-1)\epsilon^2 - \frac{1}{2}a\epsilon + \frac{1}{6}a+\frac{1}{6}} \\
& - 
\frac{(1 - \epsilon)a + \epsilon - \frac{(a-2)(1-\epsilon)}{2} }{ \frac{(1-\epsilon)^2 a + 2 \epsilon (1 - \epsilon) }{2} } \text{.} \end{align*}
is nonpositive.

When $\epsilon = 0$, the above difference simplifies to give
\begin{align*}
\mu_{1 } (L_{0};E)- \mu(L_{0}) & = \frac{-a^{2}+3a-2}{a^{2}+a} \text{.}
\end{align*}
It follows, when $\epsilon = 0$, that
$$
\mu_{1 } (L_{0};E)- \mu(L_{0})  
\begin{cases}
= 0 & \text{ if $a =  1,2$; and} \\
< 0 & \text{ if $a \not =  1, 2$.}
\end{cases}
$$

\end{example}

\begin{example}\label{ex:r:1:a:1}
We now consider Example \ref{slope:stability:FFa:surfaces} for the special case that $r = 1$ and $a = 1$.  It is instructive to consider this case in further details.  In doing so, we also illustrate the smallest instance of Theorem \ref{instability:main:thm}.   
In this case, when $a = 1$, we have that
\begin{align*}
\mu_{1 - \epsilon} (L_{\epsilon};E) - \mu(L_{\epsilon};E) & = 
\frac{\frac{\epsilon^2}{2} - \frac{3}{2} \epsilon + 1}{\frac{1}{6}\epsilon^3 - \frac{1}{2}\epsilon + \frac{2}{6}} - \frac{\epsilon - 3}{\epsilon^2 - 1}  =  \frac{2 \epsilon}{\epsilon^2 + 3 \epsilon + 2} 
\end{align*}
whence, when $a = 1$,
$$
\mu_{1 - \epsilon} (L_{\epsilon};E) - \mu(L_{\epsilon};E) \begin{cases} 
= 0 & \text{ if $\epsilon = 0$} \\
> 0 & \text{ if $\epsilon > 0$.}
\end{cases}
$$

For completeness, we consider next the case that $a = 2$.  In this case
\begin{align*}
\mu_{1 - \epsilon} (L_{\epsilon};E) - \mu(L_{\epsilon};E) & = 
\frac{\frac{\epsilon^2}{2}-\frac{3}{2}\epsilon+1}{\frac{1}{2}\epsilon^2-\epsilon+\frac{1}{2}} - \frac{2(1-\epsilon)+\epsilon}{(1-\epsilon)^2+\epsilon(1-\epsilon)}  = 0 \text{.}
\end{align*}
\end{example}

As a final introductory example, and to motivate Section \ref{Case:of:OOr:OO1} and Theorem \ref{instability:main:thm},  turning to higher dimensions, the next simplest case is the situation of $\PP^3$ blown-up along a line.  

\begin{example}[Compare with {\cite[Theorem 1]{Hashimoto:2018}}]\label{PP^3:blownup:line} 
Over the projective line $\PP^1$ consider the vector bundle
$$
\mathcal{E} := \Osh_{\PP^1}^{\oplus 2} \oplus \Osh_{\PP^1}(1)
$$
together with its associated projective bundle
$$
\pi \colon \PP(\mathcal{E}) := \operatorname{Proj} \left(\operatorname{Sym}^\bullet(\mathcal{E}) \right)\rightarrow \PP^3 \text{.}
$$
Let $\Osh_{\PP(\mathcal{E})}(1)$ be the tautological line bundle;  the morphism
$$
\PP(\mathcal{E}) \xrightarrow{|\Osh_{\PP(\mathcal{E})}(1)|} \PP^3
$$
is identified with $\PP^3$ blown-up along a line.  In particular, the tautological line bundle $\Osh_{\PP(\mathcal{E})}(1)$ is big and nef.  

Inside of the Chow group $A^\bullet(\PP(\mathcal{E}))$ let $M$ denote the class of $\Osh_{\PP(\mathcal{E})}(1)$, $N$ the class of $\pi^*\Osh_{\PP^1}(1)$ and set
$E := M - N  \text{.}$
The canonical class is then
$\K := - 3 M - N \text{.}$
As in Example \ref{slope:stability:FFa:surfaces}, for $\epsilon \in [0,1) \bigcap \QQ$, set
$$L = L_\epsilon:= M - \epsilon E \text{.}$$

In determining the extent to which the big and nef class $L_{\epsilon}$ is slope stable along $E$, the relevant numerical  quantities are
$$
\mu_{1 - \epsilon}(L_\epsilon,E) 
= 
\frac{\frac{1}{2}\epsilon^2 - \epsilon + \frac{1}{2} }{- \frac{1}{12}\epsilon^4 + \frac{1}{6}\epsilon^3 - \frac{1}{6}\epsilon + \frac{1}{12}} 
=  - \frac{6}{\epsilon^2-1}
$$
and
$$
\mu = \mu(L_\epsilon) 
= \frac{ - \frac{\epsilon^2}{2} - \frac{\epsilon}{2}+1}{\frac{\epsilon^3}{3} - \frac{\epsilon^2}{2} + \frac{1}{6} } \\
= \frac{-6\epsilon -12}{4\epsilon^2-2\epsilon-2 } 
\text{.}
$$
In particular it follows that
$$
\mu_{1 - \epsilon}  - \mu 
= \frac{3\epsilon}{2\epsilon^2+3\epsilon + 1} \geq 0 \text{.}
$$

Observe that the above calculations imply, in particular, that the big and nef class 
$M = L_0$ is slope semistable yet not slope stable along $E$ with respect to its nef threshold $c = 1$ 
whereas the ample class $L_\epsilon =  M - \epsilon E$
is not slope semistable along $E$ with respect to its nef threshold 
$c = 1 - \epsilon \text{.}$
The Seshadri slope $\K$-semistablity of $M$ along $E$ is a special instance of Theorem \ref{instability:main:thm}.  The Seshadri slope $\K$-instability of $L_{\epsilon}$ for $\epsilon \in (0,1) \bigcap \QQ$, is a special case of the main result of \cite{Hashimoto:2018}.
\end{example}

\section{Slope $\K$-semistability for  $\PP(\Osh_{\PP^s}^{\oplus r} \oplus \Osh_{\PP^s}(1))$}\label{Case:of:OOr:OO1}

We now consider in further detail the setting of Section \ref{case:of:PPs:OOr:OOa} with particular emphasis on the case that $a = 1$.   In particular, set 
$$
\mathcal{E} = \Osh_{\PP^s}^{\oplus r} \oplus \Osh_{\PP^s}(1) 
$$
and recall that the natural morphism
$$
\PP(\Osh_{\PP^s}^{\oplus r} \oplus \Osh_{\PP^s}(1)) \xrightarrow{|\Osh_{\PP(\mathcal{E})}(1)|} \PP^n\text{,} 
$$
for $n := r + s$, is identified with the blowing-up of $\PP^n$ along an $r-1$-plane.  
The class of the exceptional divisor $E$ is $M - N$.  Here $M$ is the class of $\Osh_{\PP(\mathcal{E})}(1)$ whereas $N$ is the class of the pullback of $\Osh_{\PP^s}(1)$ to $\PP(\mathcal{E})$.

In this section, our aim is to establish \emph{Seshardi slope $\K$-semistability} for the big and nef class $L_0 = M$ with respect to the exceptional divisor $E$.  In particular, this is a consequence of Theorem \ref{instability:main:thm} below which also implies the conclusion of Theorem \ref{big:and:nef:blow:up:PP:n}.

\begin{theorem}\label{instability:main:thm}
Let $\mathcal{E} = \Osh_{\PP^s}^{\oplus r} \oplus \Osh_{\PP^s}(1)$ and let $L_0 = \Osh_{\PP(\mathcal{E})}(1)$.  Then $L_0$ is Seshadri slope $\K$-semistable along the exceptional divisor $E$.  In more precise terms
\begin{equation}\label{slope:stab:key:eqn}
\mu_1(L_0;E) - \mu(L_0) = 0 \text{.}
\end{equation}
\end{theorem}

Before proving Theorem \ref{instability:main:thm} we make note of a couple of lemmas.  

\begin{lemma}\label{epsilon:zero:big:nef:slope:a0}
In the setting of Theorem \ref{instability:main:thm}
$$
\mu(L_0) = \frac{\alpha_1(L_0)}{\alpha_0(L_0)} = \frac{(n+1)n}{2} \text{.}
$$
\end{lemma}

\begin{proof}
Since $n = r + s$, the description of $\alpha_0(L_0)$ and $\alpha_1(L_0)$ follow from the description of $\alpha_0(t)$ and $\alpha_1(t)$ when $t = 0$, $\epsilon = 0$ and $a = 1$. 
For the final assertion simply note that
$$
\mu(L_0) := \frac{\alpha_1(L_0)}{\alpha_0(L_0)} = \frac{n+1}{2(n-1)!} \frac{n!}{1} = \frac{(n+1)n}{2} \text{.}
$$
\end{proof}

\begin{lemma}\label{slope:invariants:a:0}
In the setting of Theorem \ref{instability:main:thm} it holds true that
\begin{enumerate}
\item[(i)]{
$\int_0^1 \alpha_0(L_0;E,t) \mathrm{d}t = \frac{n-r+1}{(n+1)!}$
}
\item[(ii)]{
$ \int_0^1 \alpha_1(L_0;E,t) \mathrm{d}t = \frac{(n-r+1)(r+1)+(n-r)^2}{2n!}$
}
\item[(iii)]{
$\frac{1}{2} \int_0^1 \alpha_0'(L_0;E,t) \mathrm{d}t =  \frac{- 1}{2n!}$.
}
\end{enumerate}
\end{lemma}
\begin{proof}
For (i) note that
$$
\alpha_0(L_0;E,t) = \frac{1}{n!} \sum_{i=r}^n \binom{n}{i} (1-t)^i t^{n-i} \text{.}
$$
Thus
\begin{align*}
\int_0^1 \alpha_0(L_0;E,t)\mathrm{d}t &= \frac{1}{n!} \sum_{i=r}^n \frac{n!}{(n-i)!i!} \frac{i! (n-i)!}{(n+1)!} \\
& = \frac{1}{n!} \sum_{i=r}^n \frac{1}{n+1} 
= \frac{n-r+1}{(n+1)!} \text{.}
\end{align*}

For (ii)
\begin{align*}
\alpha_1(L_0;E) & = - \frac{1}{2(n-1)!} \left( -(r+1) \sum_{i=r-1}^{n-1} \binom{n-1}{i}(1-t)^i t^{n-1-i} \right. \\ 
& \left. - (n-r) \sum_{i=r}^{n-1}\binom{n-1}{i}(1-t)^it^{n-1-i}\right) \text{.}
\end{align*}
Thus
\begin{align*}
\int_0^1 \alpha_1(L_0;E,t)\mathrm{d}t &= - \frac{1}{2(n-1)!} \left( -(r+1) \sum_{i=r-1}^{n-1} \frac{(n-1)!}{i!(n-1-i)!} \frac{i!(n-1-i)!}{n!} \right. \\
& 
\left.  - (n-r) \sum_{i=r}^{n-1} \frac{ (n-1)!}{i!(n-1-i)!} \frac{i!(n-1-i)!}{n!} \right) \\
&= \frac{1}{2(n-1)!} \left( (r+1) \sum_{i=r-1}^{n-1} \frac{1}{n} + (n-r) \sum_{i=r}^{n-1} \frac{1}{n} \right) \\
& = \frac{1}{2n(n-1)!} \left( (r+1)(n-r+1) + (n-r)(n-r) \right) \\
& = \frac{ (n-r+1)(r+1) + (n-r)^2}{2n (n-1)! } \text{.}
\end{align*}

Finally, for (iii) 
\begin{align*}
\alpha_0'(L_0;E,t) &= \frac{-1}{(n-1)!} \left( \sum_{i=r-1}^{n-1} \binom{n-1}{i} (1 - t)^it^{n-1-i} \right. \\
& \left. - \sum_{i=r}^{n-1}\binom{n-1}{i} (1 - t)^i t^{n-1-i} \right) \text{.}
\end{align*}
Thus
\begin{align*}
\int_0^1 \frac{\alpha_0'(L_0;E,t)}{2} \mathrm{d}t &= \frac{-1}{2(n-1)!} \left( \sum_{i = r-1}^{n-1} \binom{n-1}{i} \frac{i!(n-1-i)!}{n!} \right. \\
& \left. - \sum_{i=r}^{n-1} \binom{n-1}{i} \frac{i!(n-1-i)!}{n!} \right) \\
&= \frac{-1}{2(n-1)!} \left( \frac{n-r+1}{n} - \frac{n-r}{n} \right) = \frac{-1}{2n(n-1)!} \text{.}
\end{align*}
\end{proof}

With these preparatory calculations at hand, we now establish Theorem \ref{big:and:nef:blow:up:PP:n} (in the form of Theorem \ref{instability:main:thm}).

\begin{proof}[Proof of Theorem \ref{instability:main:thm}]
By Lemma \ref{slope:invariants:a:0}, observe that $\mu_1(L_0;E)$ can be described as
\begin{align*}
\mu_1(L_0;E) & = \frac{(n-r+1)(r+1)+(n-r)^2 -1}{2 n!} \frac{(n+1)!}{n-r+1} \\
&= \frac{n+1}{2}\left( r+1 + \frac{(n-r)^2 - 1 }{n-r+1} \right) \text{.}
\end{align*}

Then note that by Lemma \ref{epsilon:zero:big:nef:slope:a0}
$$
\mu(L_0) = \frac{\alpha_1(L_0)}{\alpha_0(L_0)} = \frac{n+1}{2(n-1)!} \frac{n!}{1} = \frac{(n+1)n}{2}  \text{.}
$$
Further, observe that
$$
(n-r+1)(r+1) + (n-r)^2 - n(n-r+1) - 1 = 0 \text{.}
$$
The conclusion is then that
$$
\mu_1(L_0;E) - \mu(L_0)  = \frac{n+1}{2} \left(r+1 + \frac{(n-r)^2- 1}{n-r+1} \right) - \frac{(n+1)n}{2} = 0 \text{.}
$$
\end{proof}

\providecommand{\bysame}{\leavevmode\hbox to3em{\hrulefill}\thinspace}
\providecommand{\MR}{\relax\ifhmode\unskip\space\fi MR }
\providecommand{\MRhref}[2]{%
  \href{http://www.ams.org/mathscinet-getitem?mr=#1}{#2}
}
\providecommand{\href}[2]{#2}

\end{document}